\documentclass[12pt]{amsproc}
\usepackage{amssymb}
\usepackage{amsmath}
\usepackage{amsfonts}
\usepackage{geometry}
\usepackage{graphicx}
\providecommand{\U}[1]{\protect\rule{.1in}{.1in}}
\theoremstyle{plain}

\newtheorem{corollary}{Corollary}

\newtheorem{lemma}{Lemma}

\newtheorem{remark}{Remark}

\newtheorem{theorem}{Theorem}
\numberwithin{equation}{section}

\newcommand{\eps}{\epsilon}

\begin{document}
\title[Maximum principles and Symmetry results ]{The Maximum Principles and
Symmetry results  for  Viscosity Solutions of Fully Nonlinear
Equations}
\author{Guozhen Lu and Jiuyi Zhu  }
\address{Guozhen Lu\\ \\Department of Mathematics\\
Wayne State University\\
Detroit, MI 48202, USA\\
Emails: gzlu@math.wayne.edu}
\address{ Jiuyi Zhu  \\
Department of Mathematics\\
Wayne State University\\
Detroit, MI 48202, USA\\
Emails:  jiuyi.zhu@wayne.edu }
\thanks{\noindent Research is partly supported by a US NSF grant.}
\date{}
 \dedicatory{ }

\begin{abstract}
This paper is concerned about maximum principles and radial symmetry
for viscosity solutions of fully nonlinear partial differential
equations. We obtain the radial symmetry and monotonicity properties for
nonnegative viscosity solutions of
\begin{equation}
F( D^2 u)+u^p=0  \quad \quad \mbox{in} \ \mathbb R^n \label{abs}
\end{equation}
under the asymptotic decay rate $u=o(|x|^{-\frac{2}{p-1}})$ at
infinity, where $p>1$ (Theorem 1, Corollary 1). As a consequence of our symmetry results, we
obtain the nonexistence of any nontrivial and nonnegative solution
when $F$ is the Pucci extremal operators (Corollary 2). Our symmetry and
monotonicity results also apply to Hamilton-Jacobi-Bellman or Isaccs
equations.  A new maximum principle for viscosity solutions to fully
nonlinear elliptic equations is established (Theorem 2). As a result, different
forms of maximum principles on bounded and unbounded domains are
obtained. Radial symmetry, monotonicity and the corresponding
maximum principle for fully nonlinear elliptic equations in a
punctured ball are shown (Theorem 3). We also investigate the radial symmetry
for viscosity solutions of fully nonlinear parabolic partial
differential equations (Theorem 4).
\end{abstract}

\maketitle
\section{Introduction}

In studying partial differential equations, it is often of interest
to know if the solutions are radially symmetric. In this article, we
consider radial symmetry results for viscosity solutions of the
fully nonlinear elliptic equations
\begin{equation}
F( D^2 u)+u^p=0  \quad \quad \mbox{in} \ \mathbb R^n \label{equ1}
\end{equation}
and the Dirichlet boundary value problem in a punctured ball
\begin{equation}
\left \{ \begin{array}{lll}
 F( Du, D^2 u)+f(u)= 0 \quad \quad
&\mbox{in} \ \mathbb B\backslash \{0\}, \\
u=0 \quad & \mbox{on} \ \mathbb \partial B.
\end{array}
\right. \label{ball}
\end{equation}
We also obtain the radial symmetry for viscosity solutions of the fully
nonlinear parabolic equation
\begin{equation}
\left \{ \begin{array}{lll}
 \partial_t u- F( Du, D^2 u)-f(u)= 0 \quad \quad
&\mbox{in} \ \mathbb R^n\times (0, \ T], \\
u(x, 0)=u_0(x) \quad & \mbox{on} \ \mathbb R^n\times\{0\}.
\end{array}
\right. \label{equ2}
\end{equation}
We assume in the above that $F(Du, D^2 u)$ is a continuous function
defined on $\mathbb R^n\times S_n(\mathbb R)$, where  $S_n(\mathbb
R) $ is the space of real, $n\times n$ symmetric matrix, $f(u)$ is a
locally Lipschitz continuous function and the initial value $u_0(x)$
is continuous.  More precisely, we consider $F: \mathbb R^n\times
S_n(\mathbb R)\to \mathbb R$ satisfies the following structure
hypothesis.

 (\emph{F1}): There
exist $\gamma\geq 0$ and $0<\Lambda_1\leq\Lambda_2<\infty$ such that
for all $M, N\in S_n(\mathbb R)$ and $\xi_1, \xi_2\in \mathbb R^n$,
\begin{equation}
\mathcal{M}^-_{\Lambda_1, \Lambda_2}(M)-\gamma|\xi_1-\xi_2|\leq
F(\xi_1, M+N)-F(\xi_2, N)\leq \mathcal{M}^+_{\Lambda_1,
\Lambda_2}(M) +\gamma|\xi_1-\xi_2|, \label{str}
\end{equation}
where $\mathcal{M}^\pm_{\Lambda_1, \Lambda_2}$ are the Pucci
extremal operators, defined as
\begin{equation}
\mathcal{M}^+_{\Lambda_1, \Lambda_2}(M)=\Lambda_2\Sigma_{e_i>0}
e_i+\Lambda_1 \Sigma_{e_i<0}e_i, \label{puc1}
\end{equation}
\begin{equation}
\mathcal{M}^-_{\Lambda_1, \Lambda_2}(M)=\Lambda_1\Sigma_{e_i>0}
e_i+\Lambda_2 \Sigma_{e_i<0}e_i \label{puc2}
\end{equation}
where $e_i$, $i=1, \cdots, n$,  is an eigenvalue of $M$. \\

For any $M=(m_{ij})\in S_n(\mathbb R)$, let $M^{(k)}$ be the matrix
obtained from $M$ by replacing $m_{ik}$ and $m_{kj}$ by $-m_{ik}$
and $-m_{kj}$ for $i\not=k$, $j\not=k$, respectively. For any vector
$p$, let $$p^{(k)}=(p_1,\cdots, p_{k-1}, -p_k, p_{k+1}, \cdots,
p_n).$$ We assume the following hypothesis for $F$,

(\emph{F2}):
\begin{equation}
F(p^{(k)}, M^{(k)})=F(p, M) \label{equ}
\end{equation}
for $k=1, \cdots, n.$ \\

 Note that $M$ and $M^{(k)}$ have the same eigenvalues. In this
sense,
$$\mathcal{M}^{\pm}_{\Lambda_1, \Lambda_2}(M^{(k)})=\mathcal{M}^{\pm}_{\Lambda_1,
\Lambda_2}(M).$$ Under the hypotheses $(F1)$ and  $(F2)$,  it is
nature to see that the following hypotheses hold for the $F(D^2 u)$
in (\ref{equ1}), that is, \begin{equation} \mathcal{M}^-_{\Lambda_1,
\Lambda_2}(M)\leq F( M+N)-F( N)\leq \mathcal{M}^+_{\Lambda_1,
\Lambda_2}(M), \label{red1}
\end{equation}
\begin{equation}
F(M^{(k)})=F( M). \label{red2}
\end{equation}
Let $\xi_1=\xi_2$, the hypothesis $(F1)$ implies that the uniform
ellipticity for the fully nonlinear equation. Namely, there exist
$0<\Lambda_1\leq \Lambda_2<\infty$ such that
$$ \Lambda_1 tr(N)\leq F(\xi_1, M+N)-F(\xi_1, M)\leq \Lambda_2 tr(N)$$
for all $M, N\in S_n(\mathbb R), $ $N\geq 0$, where $tr(N)$ is the
trace of the matrix $N$. It is easy to see that the Pucci's
operators (\ref{puc1}), (\ref{puc2}) are extremal in the sense that
$$ \mathcal{M}^+_{\Lambda_1, \Lambda_2}(M)=\sup_{A\in\mathcal {A}_{\Lambda_1,\Lambda_2}}tr(AM), $$
$$ \mathcal{M}^-_{\Lambda_1, \Lambda_2}(M)=\inf_{A\in\mathcal {A}_{\Lambda_1,\Lambda_2}}tr(AM), $$
where $\mathcal {A}_{\Lambda_1,\Lambda_2}$ denotes the set of all
symmetric matrix whose eigenvalues lie in the interval $[\Lambda_1,
\Lambda_2]$.

The moving plane method is a powerful tool to show the radial
symmetry of solutions in partial differential equations. This method
goes back to A.D. Alexandroff  and then Serrin \cite{S} applies it
to elliptic equations for overdetermined problems. Gidas, Ni and
Nirenberg \cite{GNN1} further exploit this tool to obtain radial
symmetry of positive $C^2$ solutions of the Dirichlet boundary
problem for
$$ \triangle u+f(u)=0, \quad \quad f\in C^{0, 1}(\mathbb R)$$
in a ball. Notice that the Laplace operator corresponds to
$\Lambda_1=\Lambda_2=1$ in our $F(Du, D^2u)$. In \cite{GNN2}, Gidas,
Ni and Nirenberg extend their techniques to elliptic equations in
$\mathbb R^n$. By assuming that the solutions decay to zero at
infinity at a certain rate, the radial symmetry of positive
classical solutions is also derived. Further extensions and simpler
proofs are due to Berestycki and Nirenberg  \cite{BN} and C. Li
\cite{L}. For the detailed account and applications of the moving
plane method for semilinear elliptic equations, we refer to Chen and
Li's book \cite{CL} and references therein.

Radial symmetry results for classical solutions of fully nonlinear
elliptic equations are considered. See e.g. \cite{L} and  \cite{LN}.
Recently, Da Lio and Sirakov \cite{DS} studied the radial symmetry
for viscosity solutions of fully nonlinear elliptic equations. The
moving plane method is adapted to work in the setting of viscosity
solutions. We would like to mention that, in these quoted results
for radial symmetry in $\mathbb R^n$, a supplementary hypothesis
that $f(u)$ is nonincreasing in a right neighborhood of zero is
required. In the context of fully nonlinear equation $F(x, u, Du,
D^2u)=0$, it is equivalent to say that the operator $F$ is proper in
a right neighborhood of zero, i.e. the operator $F$ is nonincreasing
in $u$ in the case that $u$ is small.

We are particularly interested in the nonnegative viscosity
solutions of
\begin{equation} F(D^2 u)+u^p=0 \quad \quad \mbox{in} \ \mathbb R^n
\label{are}
\end{equation}
 for $p>1$.
Note that the proper assumption (that is, nonincreasing in $u$) for
fully nonlinear equation in (\ref{are}) is violated since
$f(u)=u^p$ is not nonincreasing any more. So the previous results no
longer hold for (\ref{are}). The typical models of (\ref{are}) are
the equations
\begin{equation}
\mathcal{M}^{\pm}_{\Lambda_1, \Lambda_2}(D^2u)+u^p=0 \quad \quad
\mbox{in} \ \mathbb R^n. \label{mod}
\end{equation}
It is well known that  the moving plane method and Kelvin transform
provide an elegant way of obtaining the Liouville-type theorems
 (i.e. the nonexistence of any solution) in \cite{CL1}. For (\ref{mod}),
the critical exponent for nonexistence of any viscosity solution is
still an open problem, since the Kelvin transform does not seem to
be available. Curti and Lenoi \cite{CL2} consider the nonnegative
supersolutions of (\ref{mod}), that is,
\begin{equation}
\mathcal{M}^{\pm}_{\Lambda_1, \Lambda_2}(M)+u^p\leq 0 \quad \quad
\mbox{in} \ \mathbb R^n. \label{supe}
\end{equation}
They show that the inequality (\ref{supe}) with
$\mathcal{M}^{+}_{\Lambda_1, \Lambda_2}$ has no non-trivial solution
for $1<p\leq \frac{n^\ast}{n^{\ast}-2}$, the inequality (\ref{supe})
with $\mathcal{M}^{-}_{\Lambda_1, \Lambda_2}$ has no non-trivial
solution provided $ 1<p\leq \frac{n_\ast}{n_\ast-2}$, where the
dimension like numbers are defined as
$$n^\ast=\frac{\Lambda_1}{\Lambda_2}(n-1)+1,$$
$$n_\ast=\frac{\Lambda_2}{\Lambda_1}(n-1)+1.$$
In order to understand the solution structure for (\ref{mod}),
Felmer and Quass \cite{FQ} consider (\ref{mod}) in the case of
radially symmetric solutions. Using phase plane analysis, they
establish that

{\bf Theorem A} (i): For  (\ref{mod}) with the Pucci extremal
operator $\mathcal{M}^{+}_{\Lambda_1, \Lambda_2}$, there exists no
non-trivial radial solution if $1<p<p^\ast_+$ and $n^\ast>2$, where
$$\max\{\frac{n^\ast}{n^\ast-2}, \,
\frac{n+2}{n-2}\}<p^\ast_+<\frac{n^\ast+2}{n^\ast-2}.$$ \quad (ii):
 For  (\ref{mod}) with the Pucci  extremal
operator $\mathcal{M}^{-}_{\Lambda_1, \Lambda_2}$, there exists no
non-trivial radial solution if $1<p<p^\ast_-$, where
$$\frac{n_\ast+2}{n_\ast-2}<p^\ast_-<\frac{n+2}{n-2}. $$

An explicit expression for $ p^\ast_+, \ p^\ast_-$ in term of
$\Lambda_1, \Lambda_2, n$ are still unknown. In order to obtain the
full range of  the exponent $p$ for the Liouville-type theorem in
(\ref{mod}), it is interesting to prove that the solutions in
(\ref{mod}) are radially symmetric.

We first consider the radial symmetry for the fully nonlinear
equations with general operator $F(D^2u)$ and show that

\begin{theorem}
Assume  $F(D^2u)$ satisfies  (\ref{red1}) and (\ref{red2}). Let
$n^\ast>2$. If $u\in C(\mathbb R^n)$ is a nonnegative non-trivial
solution of (\ref{equ1}) and
\begin{equation}
u=o(|x|^{-\frac{2}{p-1}}) \quad \mbox{as} \ |x|\to\infty \label{asp}
\end{equation}
for $p>1$, then $u$ is radially  symmetric and strictly decreasing
about some point. \label{th1}
\end{theorem}

In the same spirit of the proof in Theorem \ref{th1}, our
conclusions also hold for general function $f(u)$, i.e.
\begin{equation}
F(D^2 u)+f(u)=0 \quad \quad \mbox{in} \ \mathbb R^n. \label{gene}
\end{equation}
\begin{corollary}
Assume that $F(D^2 u)$ satisfies (\ref{red1}) and (\ref{red2}), and
\begin{equation}
\frac{f(u)-f(v)}{u-v}\leq c(|u|+|v|)^\alpha, \ \mbox{for} \ u, v \
\mbox{sufficiently small, and some }  \alpha>\frac{2}{n^\ast-2} \mbox{and} \ c>0.
\label{assf}
\end{equation}
Let $n^\ast>2$ and $u$ be a positive solution of (\ref{gene}) with
\begin{equation}
u(x)=O(|x|^{2-n^\ast})
\end{equation}
at infinity.  Then $u$ is radially symmetric and strictly decreasing
about some point in $\mathbb R^n$. \label{coro}
\end{corollary}

Once the radial symmetry property of solutions is established, with
the help of Theorem A, we immediately have the following corollary.
We hope that our symmetry results shed some light on the complicated
problem of Liouville-type theorems in (\ref{mod}) for the  full
range of the exponent $p$.

\begin{corollary}
(i) For  (\ref{mod}) with the Pucci extremal operator
$\mathcal{M}^{+}_{\Lambda_1, \Lambda_2}$ , there exists no
non-trivial nonnegative solution satisfying (\ref{asp}) if $1<p<p^\ast_+$ and $n^\ast>2$. \\
(ii) For  (\ref{mod}) with the Pucci extremal operator
$\mathcal{M}^{-}_{\Lambda_1, \Lambda_2}$, there exists no
non-trivial nonnegative solution satisfying (\ref{asp})  if
$1<p<p^\ast_-$ and $n^\ast>2$.
\end{corollary}

In carrying out the moving plane method, the maximum principle plays
a crucial role. In order to adapt the moving plane method to
non-proper fully nonlinear equations (that is, $F(x, u, Du, D^2u)$
is not nondecreasing in $u$), a new maximum principle has to be
established for viscosity solutions.

In this paper,
we will establish  a new maximum principle for viscosity solutions to the equation
$$
\mathcal{M}_{\Lambda_1, \Lambda_2}^-(D^2 u)-\gamma|Du|+c(x)u\leq 0
\quad \mbox{in} \ \Omega,$$ where $c(x)\in L^\infty(\Omega)$ is not
necessarily  negative. Similar maximum principle for classical solutions to semilinear equations was given in \cite{CL}. Since we consider the viscosity solutions
here instead of classical solutions in \cite{CL}, considerably more  difficulties have
to be taken care of in our case. Unlike the pointwise argument in \cite{CL}, we
apply the Hopf lemma for viscosity solutions in those minimum
points.    More
specifically, we have

\begin{theorem}
Let $\Omega$ be a bounded domain. Assume that $\lambda(x), c(x)\in
L^{\infty}(\Omega), \gamma\geq 0$, and $\psi\in C^2(\Omega)\cap
C^1(\bar\Omega)$ is a positive solution in $\bar\Omega$ satisfying
\begin{equation}
\mathcal{M}_{\Lambda_1, \Lambda_2}^+(D^2\psi)+\lambda(x)\psi\leq 0.
\label{aux}
\end{equation}
Let $u$ be a viscosity solution of
\begin{equation}
\left \{ \begin{array}{lll} \mathcal{M}_{\Lambda_1, \Lambda_2}^-(D^2
u)-\gamma|Du|+c(x)u\leq 0  \quad &\mbox{in} \ \Omega,\\
u\geq 0 \quad &\mbox{on} \ \partial\Omega. \label{key}
\end{array}
\right.
\end{equation}
If
\begin{equation}
c(x)\leq \lambda(x)-\gamma|D\psi|/\psi \label{est},
\end{equation}
then $u\geq 0$ in $\Omega$. \label{max}
\end{theorem}

Note that the function $c(x)$ may not be needed to be negative in
order that the maximum principle holds. We also would like to point
out the $\psi$ is a supersolution of the equation involving
$\mathcal{M}_{\Lambda_1, \Lambda_2}^+(D^2\psi)$ instead of
$\mathcal{M}_{\Lambda_1, \Lambda_2}^-(D^2\psi)$. If a specific
$\psi(x)$ is chosen, we can get the explicit control for $c(x)$ in
order to obtain the maximum principle for (\ref{key}). We are also
able to extend the maximum principle to unbounded domains. We refer
to Section 3 for more details.

Recently Caffarelli, Li and Nirenberg \cite{CLN} \cite{CLN1}
investigated the following problem
\begin{equation}
\left \{ \begin{array}{lll} \triangle u+f(u)=0 \quad \quad
&\mbox{in}
\  \mathbb B\backslash \{0\}, \\
u=0 \quad & \mbox{in} \  \partial \mathbb B
\end{array}
\right.
\end{equation}
in the case that $f$ is locally Lipschitz. They obtained the radial
symmetry and monotonicity property of solutions using an idea of
Terracini \cite{T}. Their results are also extended to fully
nonlinear equations $F(x, u, Du, D^2u)=0$ with differentiable
components for $u\in C^2(\bar {\mathbb B}\backslash \{0\})$.
However, this prevents us from applying these results to important
classes of equations such as equations involving Pucci's extremal
operators, Hamilton-Jacobi-Bellman or Isaacs equations. A maximum
principle in a punctured domain is established in \cite{CLN}  in
order to apply the moving plane technique. However, their maximum
principle only holds in sufficiently small domains, since sufficient
smallness of the domain is used in the spirit of
Alexandroff-Pucci-Belman maximum principle (see \cite{BN}).

In this paper,
we consider
$$
\mathcal{M}_{\Lambda_1, \Lambda_2}^-(D^2 u)-\gamma|Du|+c(x)u\leq 0
\quad \ \mbox{in} \ \Omega\backslash \{0\}.$$ We will obtain  a new
maximum principle in terms of the assumption of $c(x)$ (see Lemma
\ref{sma}). It is especially true for a sufficiently small domain
just as Caffarelli, Li and Nirenberg's maximum principle, since the
bound of $c(x)$ in Lemma \ref{sma} preserves automatically if $|x|$
is small enough. Our result is not only an extension for viscosity
solutions, but also an interesting result for semilinear elliptic
equations. Furthermore, we obtain the radial symmetry of solutions
in a punctured ball.
\begin{theorem}
Let $u\in C(\mathbb B\backslash \{0\})$ be a positive viscosity
solution of (\ref{ball}) in the case that $f(u)$ is locally Lipschitz.
Then $u$ is radially symmetric with respect to the origin and $u$ is
strictly decreasing in $|x|$. \label{th2}
\end{theorem}

Finally,  we consider the radial symmetry of the Cauchy problem for
viscosity solutions of the fully nonlinear parabolic equation
(\ref{equ2}). C. Li \cite{L} obtain the monotonicity and radial
symmetry properties of classic solution $u\in C^2 (\mathbb R^n\times
(0, T])$ for fully nonlinear parabolic equations $\partial_t u-F(x, u, Du, D^2u)=0$ with
differentiable components. Again this result does not apply to fully
nonlinear parabolic equations involving  Pucci extremal operators,
Hamilton-Jacobi-Bellman or Iassac equations.
 For further extensions  about asymptotic symmetry or
radial symmetry of entire solutions, etc. for parabolic problems on
bounded or unbounded domains, we refer to the survey of
Pol$\acute{a}\check{c}$ik \cite{P}.

In this paper, we prove that
\begin{theorem}
Let $u\in C(\mathbb R^n\times (0, T])$ be a positive viscosity
solution of   (\ref{equ2}). Assume that
\begin{equation}
 |u(x,t)|\to 0 \quad \mbox{uniformly as}  \ |x|\to\infty,
 \label{inft}
\end{equation}
and
$$u_0(x_1, x')\leq u_0(y, x')  \ \mbox{for}  \ x_1\leq y\leq
-x_1, \ x_1\leq 0 \ \mbox{and} \ x'=(x_2,\cdots,x_n).$$ Then $u$ is
nondecreasing in $x_1$ and $u(x_1, x', t)\leq u(-x_1, x', t)$ for
$x_1\leq 0$. Furthermore, if $u_0(x)$ is radially symmetric with
respect to the origin and nonincreasing in $|x|$, then $u(x)$ is radial
symmetry
 with respect to $(0,t)$ for each fixed $t\in (0, T]$ and nonincreasing in $|x|$.
 \label{th3}
\end{theorem}

The outline of the paper is as follows. In Section 2, we present the
basic results for the definition of viscosity solutions, the strong
maximum principle and the maximum principle in a small domain for
viscosity solutions, etc. Section 3 is devoted to providing the
proof of Theorem \ref{th1} and Theorem \ref{max}. New maximum
principles and their extensions are established. The radial symmetry
of solutions in a punctured ball and the corresponding maximum
principle are obtained in Section 4. In Section 5, we prove the
radial symmetry for viscosity solutions of fully nonlinear parabolic
equations. Throughout the paper, The letters $C$, $c$ denote generic
positive constants, which is independent of $u$ and may vary from
line to line.

\section{Preliminaries}

In this section we collect some basic results which will be applied
 through the paper for fully nonlinear partial differential equations. We refer to
 \cite{CC}, \cite{CCKS}, \cite{CIL}, and  references therein for a detailed account.

Let us recall the notion of viscosity sub and supersolutions of the
fully nonlinear elliptic equation
\begin{equation}
F(Du, D^2u)+f(u)=0 \quad \ \mbox{in} \ \Omega, \label{def}
\end{equation}
where $\Omega$ is an open domain in $\mathbb R^n$ and $F: \mathbb
R^n\times S_n(\mathbb R)\to \mathbb R$ is a continuous map with
$F(p,M)$ satisfying \emph{(F1)}.

\emph{Definition:} A continuous function $u: \Omega \to \mathbb R$
is a viscosity supersolution (subsolution) of (\ref{def}) in $\Omega$,
when the following condition holds: If $x_0\in\Omega$, $\phi\in
C^2(\Omega)$ and $u-\phi$ has a local minimum (maximum) at $x_0$,
then
$$ F( D\phi(x_0), D^2\phi(x_0))+f(u(x_0))\leq (\geq) 0. $$

If $u$ is  a viscosity supersolution (subsolution), we say that $u$
verifies
$$F(Du, D^2u)+f(u)\leq(\geq) 0$$
in the viscosity sense. We say that $u$ is a viscosity solution of
(\ref{def}) when it simultaneously is a viscosity subsolution and
supersolution.

We also present the notion of viscosity sub and supersolutions of
the fully nonlinear parabolic equation (see e.g. \cite{W})
\begin{equation}
\partial_t u-F(Du, D^2u)-f(u)=0 \quad \ \mbox{in} \ \Omega_T:= \Omega\times(0, T]. \label{def1}
\end{equation}

\emph{Definition:} A continuous function $u: \Omega_T \to \mathbb R$
is a viscosity supersolution (subsolution) of (\ref{def1}) in
$\Omega_T$, when the following condition holds: If $(x_0,
t_0)\in\Omega_T$, $\phi\in C^2(\Omega_T)$ and $u-\phi$ has a local
minimum (maximum) at $(x_0, t_0)$, then
$$\partial_t \phi(x_0,
t_0)- F( D\phi(x_0, t_0), D^2\phi(x_0, t_0))-f(u(x_0, t_0))\geq
(\leq) 0. $$ We say that $u$ is viscosity solution of (\ref{def1})
when it both is a viscosity subsolution and supersolution.

We state  a strong maximum principle and the Hopf lemma for
non-proper operators in fully nonlinear elliptic equations (see e.g.
\cite{BD}).
\begin{lemma}
Let $\Omega\subset \mathbb R^n$ be a smooth domain and let $b(x), c(x)
\in L^{\infty}(\Omega)$. Suppose $u\in C(\bar\Omega)$ is a viscosity
 solution of
\begin{equation}
\left \{ \begin{array}{lll}
\mathcal{M}_{\Lambda_1, \Lambda_2}^-(D^2 u) -b(x)|Du|+c(x)u \leq 0
\quad &\mbox{in} \ \Omega, \nonumber \\
u\geq 0    \quad &\mbox{in} \ \Omega.
  \end{array}
  \right.
\end{equation}
Then either $u\equiv 0$ in $\Omega$ or $u>0$ in $\Omega$. Moreover,
at any point $x_0\in \partial\Omega$ where $u(x_0)=0$, we have
\begin{equation}
 \liminf_{t \to 0}\frac{u(x_0+t \nu)-u(x_0)}{t}<0,
 \nonumber
\end{equation}
where $\nu\in \mathbb R^n\backslash \{0\}$ is such that $\nu\cdot
n(x_0)>0$ and $n(x_0)$ denotes the exterior normal to $\partial
\Omega$ at $x_0$. \label{hop}
\end{lemma}
It is straightforward  to deduce the strong maximum principle for
proper operators in fully nonlinear elliptic  equations from the
Hopf Lemma.
\begin{lemma}
Let $\Omega \in \mathbb R^n$ be an open set and let $u\in C(\Omega)$ be a viscosity solution of
$$ \mathcal{M}_{\Lambda_1, \Lambda_2}^-(D^2 u) -b(x)|Du|+c(x)u \leq 0
$$
with $b(x), c(x)\in L^{\infty}(\Omega)$ and $c(x)\leq 0$. Suppose
that $u$ achieves a non-positive  minimum in $\Omega$. Then $u$ is a
constant. \label{str}
\end{lemma}

We shall make use of the following maximum principle which does not
depend on the sign of $c(x)$, but instead, on the measure of the
domain $\Omega$ (see e.g. \cite{DS}).

\begin{lemma}
Consider a bounded domain $\Omega$ and assume that $|c(x)|<m$ in
$\Omega$ and $\gamma\geq 0$. Let $u\in C(\bar\Omega)$ be a viscosity
solution of
\begin{equation}
\left \{ \begin{array}{lll} \mathcal{M}_{\Lambda_1, \Lambda_2}^-(D^2
u) -\gamma|Du|+c(x)u \leq 0
\quad &\mbox{in} \ \Omega, \nonumber \\
u\geq 0    \quad &\mbox{on} \ \partial\Omega.
  \end{array}
  \right.
\end{equation}
Then  there exists a constant
$\delta=\delta(\Lambda_1,\Lambda_2,\gamma, n, m, diam(\Omega))$ such
that   we have $u\geq 0$ in $\Omega$ provided $|\Omega|<\delta$.
\label{min}
\end{lemma}

The following result is concerned about the regularity of viscosity
solutions in \cite{CC}.
\begin{lemma}
Let $\Omega\subset \mathbb R^n$  be a bounded domain and assume that
$(F1)$ is satisfied and $f$ is locally Lipschitz. Let $u\in
C(\bar\Omega)$ be a viscosity solution of
$$ F(Du, D^2u)+f(u)=0 \quad \mbox{in} \ \Omega.  $$
Then $u$ is in $C^{1, \alpha}_{loc}(\Omega)$ for some $\alpha\in (0, 1)$.
\label{regu}
\end{lemma}

In the process of employing the moving plane method, we need to
compare $u$ at $x$ with its value at its reflection point of $x$.
The next lemma shows that the difference of a supersolution and a
subsolution of the fully nonlinear equation is still a
supersolution. Unlike the case of the classical solutions of fully
nonlinear equations $F(x, u, Du, D^2u)=0$ with differentiable
components, the difficulty here is the lack of regularity of $u$.
The following result is first showed in \cite{DS}. We also refer the
reader to \cite{CMS} and  \cite{MQ} for related results. In Section
4, we will derive similar results for viscosity solutions of fully
nonlinear parabolic equations.

\begin{lemma}
Assume that $F(Du, D^2u)$ satisfies $(F1)$ and $f$ is locally
Lipschitz. Let $u_1\in C(\bar\Omega)$ and $u_2\in C(\bar\Omega)$ be
respectively a viscosity subsolution and supersolution of
$$ F(Du, D^2u)+f(u)=0 \quad \quad \mbox{in} \ \Omega. $$
Then the function $v=u_2-u_1$ is a viscosity solution of
$$\mathcal{M}^{-}_{\Lambda_1, \Lambda_2}(D^2
v)-\gamma |Dv|+ c(x)v(x)\leq 0,$$ where
\begin{equation}
 c(x)= \left\{ \begin{array}{lll}
\frac{f(u_2(x))-f(u_1(x))}{u_2(x)-u_1(x)},  \quad \ &\mbox{if} \
u_2(x)\not= u_1(x), \\
\\
0, \quad &\mbox{otherwise}.
\end{array}
\right.
\end{equation}
\label{diff}
\end{lemma}

In the proof of Lemma \ref{diff} in \cite{DS}, an equivalent
definition of viscosity solutions in terms of semijets is used (see
\cite{CIL}). In order  to obtain the parabolic version of Lemma
\ref{diff},  we denote by $\mathcal {P}_{\Omega}^{2, +}, \mathcal
{P}_{\Omega}^{2, -} $ the parabolic semijets.

\emph{Definition:}
\begin{equation}
\begin{array}{lll}
\mathcal {P}_{\Omega}^{2, +}u(z, s)=&\{(a, p, X)\in \mathbb R\times
\mathbb R^n \times S_n(\mathbb R): u(x, t)\leq u(z, s)+a(t-s)+<p,
x-z>+ \\
\\{}&\frac{1}{2}<X(x-z), x-z> +o(|t-s|+|x-z|^2) \ \mbox{as} \
\Omega_T\ni(x, t)\to (z,s)\}.
\end{array}
\end{equation}
 While we define ${P}_{\Omega}^{2, -}(u):=-{P}_{\Omega}^{2, +}(-u)$.

\section{Symmetry of Viscosity Solutions in  $\mathbb R^n$}

In this section, we will obtain the radial symmetry of nonnegative
solution in (\ref{equ1}). We first present a technical lemma about
the eigenvalue of a radial function. It could be verified by a direct
calculation.
\begin{lemma}
Let $\psi: (0,\ +\infty)\to \mathbb R$ be a $C^2$ radial function.
For $\forall \ x\in \mathbb R^n\backslash\{0\}$, the eigenvalues of
 $D^2 \psi(|x|)$ are $\psi{''}(|x|)$, which is simple and
$\frac{\psi'(|x|)}{|x|}$, which has multiplicity $(n-1)$.
\end{lemma}
Based on the above conclusion, we may select specific functions. For
instance, let $\psi=|x|^{-q}$ and  $0<q<n^\ast-2$. Recall that
$n^\ast=\frac{\Lambda_1}{\Lambda_2}(n-1)+1$.  The eigenvalues are
$q(q+1)|x|^{-q-2}$ and $-q|x|^{-q-2}$. From the above lemma, for
$x\in \mathbb R^n\backslash\{0\}$,
\begin{equation}
\begin{array}{lll}
\mathcal{M}_{\Lambda_1, \Lambda_2}^+(D^2\psi)(x)&=&\Lambda_2
q(q+1)|x|^{-q-2}-\Lambda_1(n-1)q|x|^{-q-2} \nonumber\\
\\
&=&\frac{q(\Lambda_2(q+1)-\Lambda_1(n-1))}{|x|^2}\psi(x).
\label{spe}
\end{array}
\end{equation}
Notice that  $0<q<n^\ast-2$ implies that
$$ q(\Lambda_2(q+1)-\Lambda_1(n-1))<0.$$

We shall make use of a simple lemma, which enables us to consider the
product of a viscosity solution and an auxiliary function. The argument
is in the spirit of Lemma 2.1 in \cite{DS}. However, the idea behind
it is different. In their lemma, $u(x)$ is assumed to be
nonnegative. We do not impose this assumption. In  other words,
we specifically focus on the points where $u(x)$ is negative.

\begin{lemma}
Let $u\in C(\Omega)$ satisfy
\begin{equation}
\mathcal{M}_{\Lambda_1, \Lambda_2}^-(D^2 u)-b(x)|Du|+c(x)u\leq 0
\end{equation}
where $b(x), c(x)\in L^{\infty}(\Omega)$. Suppose $\psi\in
C^2(\Omega)\cap C^1(\bar\Omega)$ is strictly positive in
$\bar\Omega$. Assume $u(x_0)<0$. Then  $ \bar u:=u/{\psi}$ satisfies
\begin{equation}
\mathcal{M}_{\Lambda_1, \Lambda_2}^-(D^2\bar u)- \bar b (x)|D\bar
u|+\bar c(x)\bar u\leq 0
\end{equation}
at $x_0$, where $$\bar b(x)=\frac{2\sqrt{n}\Lambda_2
|D\psi|}{\psi}+|b|$$ and $$\bar
c(x)=c(x)+\frac{\mathcal{M}_{\Lambda_1, \Lambda_2}^+(D^2
\psi)+|b||D\psi|}{\psi}.$$ \label{tec}
\end{lemma}

\begin{proof}
Let $\phi(x)\in C^2(\Omega)$ be the test function that toughes $\bar
u$ from below at $x_0$, that is $\phi(x_0)=\bar u(x_0)$ and $\bar
u(x)\geq \phi(x)$ in $\Omega$. Then $u(x_0)=\phi(x_0)\psi(x_0)$ and
$u(x)\geq \phi(x)\psi(x)$ in $\Omega$, which indicates that
$\phi(x)\psi(x)$ toughes $u$ from below. Simple  calculations show
that
$$ D(\phi\psi)=D\phi\cdot \psi+ D\psi\cdot \phi,$$
$$ D^2(\phi\psi)=\phi D^2\psi+2D\phi\otimes D\psi+D^2\phi\psi,$$
where $\otimes$ denotes the symmetric tensor product with $p\otimes
q=\frac{1}{2}(p_i q_j+p_jq_i)_{i,j}$. By the properties of the Pucci
extremal operators, we have
$$\mathcal{M}_{\Lambda_1, \Lambda_2}^-(M+N)\geq
\mathcal{M}_{\Lambda_1, \Lambda_2}^-(M)+\mathcal{M}_{\Lambda_1,
\Lambda_2}^-(N),$$ $$ \mathcal{M}_{\Lambda_1,
\Lambda_2}^-(aM)=a\mathcal{M}_{\Lambda_1, \Lambda_2}^+(M)$$ for
$a\leq 0$. We also note that
$$tr(A(p\otimes q))\leq |A||p\otimes q|\leq \sqrt{n}\Lambda_2|p||q|,$$
where $A$ is a matrix whose eigenvalues lie in $[\Lambda_1,
\Lambda_2]$ and $|A|:=\sqrt{tr(A^T A)}$. Since $\phi\psi$ is a test
function for $u$ and $\phi(x_0)=\bar u(x_0)<0$, taking into account
the above properties, we get
\begin{equation}
\begin{array}{lll}
0&\geq& c(x)\phi\psi-b|D(\phi\psi)|+\mathcal{M}_{\Lambda_1,
\Lambda_2}^-(D^2( \psi\phi))
\nonumber \\
\\
&\geq &
c(x)\phi\psi+|b||D\psi|\phi-|b||D\phi|\psi+\psi\mathcal{M}_{\Lambda_1,
\Lambda_2}^-(D^2\phi)-2\sqrt{n}\Lambda_2|D\phi||D\psi| \\
\\&&+\phi\mathcal{M}_{\Lambda_1,
\Lambda_2}^+(D^2\psi) \nonumber \\
\\
&\geq&(c(x)\psi+\mathcal{M}_{\Lambda_1,
\Lambda_2}^+(D^2\psi)+|b||D\psi|)\phi-(2\sqrt{n}\Lambda_2|D\psi|+|b|\psi)|D\phi|+\psi
\mathcal{M}_{\Lambda_1, \Lambda_2}^-(D^2\phi)
\end{array}
\end{equation}
at $x_0$. Dividing both sides by $\psi$, we obtain
$$
\mathcal{M}_{\Lambda_1, \Lambda_2}^-(D^2\phi)(x_0)- \bar b
(x_0)|D\phi|(x_0)+\bar c(x_0)\phi(x_0)\leq 0, $$ where $\bar b(x),
\bar c(x)$ are in the statement of the lemma.

\end{proof}

 Using the above lemma and the strong maximum principle in Lemma
 \ref{str}, we are able to consider the maximum principle in terms of $c(x)$  for
 non-proper operators in fully
 nonlinear elliptic equation.

\begin{proof}[Proof of Theorem \ref{max}]
We prove it by contradiction argument. Suppose that $u(x)<0$
somewhere in $\Omega$. Let $$\bar u(x)=\frac{u(x)}{\psi(x)}.$$ Then
$\bar u(x)<0$ somewhere in $\Omega$. Since $u(x)\geq 0$ on $\partial
\Omega$, we may assume that $\bar u(x^\ast)=\inf_{\Omega} \bar
u(x)<0$, where $x^\ast\in \Omega$. By the continuity of $\bar u(x)$,
we can find a connected neighborhood $\Omega'$ containing $ x^\ast$
such that $\bar u(x)<0$ in $\Omega'$ and $ \bar u(x)\not \equiv
u(x^\ast)$ in $\Omega'$. Otherwise, $u(x)\equiv u(x^\ast)$ in
$\Omega$, it is obviously a contradiction. Thanks to Lemma \ref{tec}
with $b(x)$ replaced by $\gamma$, $\bar u$ satisfies
\begin{equation}
\mathcal{M}_{\Lambda_1, \Lambda_2}^-(D^2\bar u)- \bar b (x)|D\bar
u|+\bar c(x)\bar u\leq 0 \quad \mbox{in} \ \Omega'.
\end{equation}
Recall that $$\bar b(x)=\frac{2\sqrt{n}\Lambda_2
|D\psi|}{\psi}+\gamma\in L^{\infty}(\Omega')$$ and $$\bar
c(x)=c(x)+\frac{\mathcal{M}_{\Lambda_1, \Lambda_2}^+(D^2
\psi)+\gamma|D\psi|}{\psi}\in L^{\infty}(\Omega').$$ By the
assumptions (\ref{est}) and (\ref{aux}),
$$ c(x)+\frac{\mathcal{M}_{\Lambda_1, \Lambda_2}^+(D^2
\psi)+\gamma|D\psi|}{\psi}\leq 0.$$ Thanks to the strong maximum
principle in Lemma \ref{str}, $\bar u(x)\equiv \bar u(x^\ast)$ in
$\Omega'$. It contradicts our assumption. This contradiction leads
to the proof of the lemma.
\end{proof}
\begin{remark}

1. From the proof, we can see that the same reasoning follows when
the condition (\ref{aux}) and (\ref{est}) hold where $u$ is
negative.

2. If $c(x), \lambda(x)$ are continuous, we only need
$c(x^\ast)<\lambda(x^\ast)-\gamma |D\psi|/\psi (x^\ast)$, where
$x^\ast$ is the point where $u$ reaches minimum. \label{rem1}
\end{remark}

In the spirit of the above argument, we extend the corresponding
maximum principle to unbounded domains. We need to guarantee that the
minimum is only achieved in the interior of the domain.

\begin{lemma}
Let $\Omega$ be an unbounded domain. If $u, \psi$ satisfy the same
conditions as that in Theorem \ref{max} and assume that
\begin{equation}
\liminf_{|x|\to\infty}\frac{u(x)}{\psi(x)}\geq 0 \label{inf},
\end{equation}
then $u\geq 0$ in $\Omega$. \label{con}
\end{lemma}
\begin{proof}
Note that the assumption (\ref{inf}) implies that  the minimum of
$u/\psi$ will not go to infinity. Then the minimum of $u/\psi$  lies
only in the interior of $\Omega$. Applying the same argument as in the
proof of Theorem \ref{max}, the conclusion follows.
\end{proof}

If some particular $\psi(x)$ is given, then $c(x)$ could be
controlled explicitly, which is especially useful in applying the
maximum principle. We call the following useful maximum principle as
``Decay at infinity".

\begin{corollary}
(Decay at infinity) Assume that there exists $R>0$ such that
\begin{equation}
c(x)\leq
\frac{-q(\Lambda_2(q+1)-\Lambda_1(n-1))}{|x|^2}-\frac{\gamma q}{|x|}
\quad \quad \mbox{for} \ |x|>R \label{dec}
\end{equation}
and
\begin{equation}
\liminf_{|x|\to\infty} u(x)|x|^q\geq 0. \label{cay}
\end{equation}
Let $\Omega$ be a region in $\mathbb B^c_R(0)=\mathbb R^n\backslash
\mathbb B_R(0)$. If $u$ satisfies (\ref{key}) in $\Omega$, then
$$u(x)\geq 0   \quad \quad \mbox{for all } \ x\in\Omega.$$
\end{corollary}
\begin{proof}
We consider the specific function $\psi(x)=|x|^{-q}$, As we know,
$$ \mathcal{M}_{\Lambda_1, \Lambda_2}^+(D^2\psi)(x)-\frac{q(\Lambda_2
(q+1)-\Lambda_1(n-1))}{|x|^2}\psi(x)=0. $$ Applying Lemma \ref{con},
we conclude the proof.
\end{proof}

\begin{remark}
i)It is similar to Remark \ref{rem1}, the conclusion holds when
(\ref{dec}) is true at points where $u$ is negative.

ii) In the case of $\gamma=0$, $c(x)\leq
\frac{-q(\Lambda_2(q+1)-\Lambda_1(n-1))}{|x|^2}$. Notice that $c(x)$
may not be needed to be negative in order that the maximum principle
holds. \label{rem2}
\end{remark}

In the rest of this section, we are going to adapt the moving plane
technique in  the viscosity solution setting to prove Theorem
\ref{th1}. We refer to the book \cite{CL} for more account of the
moving plane method in semilinear elliptic equations. Before we
carry out the moving plane method, we introduce several necessary
notations. Set
$$\Sigma_{\lambda}=\{x=(x_1,\cdots, x_n)\in \mathbb R^n|
x_1<\lambda\}$$ and $T_\lambda=\partial\Sigma_{\lambda}$. Define
$x^{\lambda}$ be the reflection of $x$ with respect to $T_\lambda$,
i.e. $x^\lambda=(2\lambda-x_1, x_2,\cdots, x_n)$. Let
$$u_\lambda(x)=u(x^{\lambda})$$ and
$$v_\lambda(x)=u_\lambda(x)-u(x).$$ The moving plane method to obtain the radial symmetry
 consists of two steps.
In the first step, we show that the plane can move, that is, we will
deduce  that, for sufficiently negative $\lambda$,
\begin{equation}
v_{\lambda}(x)\geq 0, \quad \quad \forall x \ \in \Sigma_\lambda,
\label{pla}
\end{equation}
where we are going to use the corollary of decay at infinity. In the
second step, we will move the plane $T_\lambda$ to the right as long
as (\ref{pla}) holds. The plane will stop at some critical position,
say at $\lambda=\lambda_0$. We will verify that
\begin{equation}
v_{\lambda_0}\equiv 0, \quad \quad \forall x\in\Sigma_{\lambda_0}.
\label{cri}
\end{equation} These two steps imply that
$u(x)$ is symmetric and monotone decreasing about the plane
$T_{\lambda_0}$. Since the equation (\ref{equ1}) is invariant under
rotation, we can further infer that $u(x)$ must be radially symmetric
with respect to some point.

\begin{proof} [Proof of Theorem \ref{th1}:]
We derive the proof in two steps.

\emph{Step 1}: By the hypothesis (\ref{red2}), $u_\lambda$ satisfies
the same equation as $u$ does. Thanks to Lemma \ref{diff} for the
case $\gamma=0$,
\begin{equation}
\mathcal{M}^{-}_{\Lambda_1, \Lambda_2}(D^2
v_{\lambda})+p\psi_{\lambda}^{p-1}(x)v_{\lambda}(x)\leq 0,
\label{sub}
\end{equation}
where $\psi_{\lambda}(x)$ is between $u_{\lambda}(x)$ and $u(x)$. In
order to apply the corollary of decay at infinity, by $(ii)$ in
Remark \ref{rem2}, it is sufficient to verify that
\begin{equation}\psi_{\lambda}^{p-1}(x)\leq
\frac{C}{|x|^2} \label{cccw}
\end{equation} and
\begin{equation}
\liminf_{|x|\to\infty}v_{\lambda}(x)|x|^q\geq 0. \label{lim}
\end{equation}

For (\ref{cccw}), to be more precise, we only need to show that
(\ref{cccw}) holds at the points $\tilde{x}$ where $v_{\lambda}$ is
negative (see Remark \ref{rem2}). At those points,
$$u_{\lambda}(\tilde{x})< u (\tilde{x}).$$
Then
$$0\leq u_{\lambda}(\tilde{x})\leq \psi_{\lambda}(\tilde{x})\leq u
(\tilde{x}).$$ By the decay assumption (\ref{asp}), we derive that
$$c(\tilde{x})=p\psi_{\lambda}^{p-1}(\tilde{x})\leq
o(|\tilde{x}|^{-2})\leq C|\tilde{x}|^{-2}, $$ that is, (\ref{cccw})
is satisfied. Note that the fact   $\lambda$ is sufficiently close
to negative infinity is applied. By the decay assumption (\ref{asp})
again, for any small $\eps$,
$$\liminf_{|x|\to\infty}v_{\lambda}(x)|x|^q\geq
\liminf_{|x|\to\infty}-u(x)|x|^q\geq
\liminf_{|x|\to\infty}\frac{-\epsilon}{|x|^{\frac{2}{p-1}-q}} .$$ If
$\frac{2}{p-1}-q>0$, then (\ref{lim}) is
fulfilled. Hence we fixed $ 0<q<\min \{ \frac{2}{p-1},
(n^\ast-2)\}$.
\\

\emph{Step 2}: We continue to move the plane $T_\lambda$ to the
right as long as (\ref{pla}) holds. Define
$$\lambda_0=\sup\{\lambda \ |v_\mu(x)\geq 0 \ \mbox{in} \ \Sigma_\mu \ \mbox{for every }\ \mu\leq \lambda\}. $$
Since $u(x)\to 0$ as $|x|\to\infty$, we infer that $\lambda_0<
\infty$. If $\lambda_0>0$, by the translation invariance of the
equation, we may do a translation to let the critical position be
negative. If $\lambda_0=0$, we move the plane from the positive
infinity to the left. If $\lambda_0=0$ again, we obtain the symmetry
of the solution at $x_1=0$. In all the cases, we may consider
$\Sigma_{\lambda_0}$ with $\lambda_0<0$, which avoids the
singularity of $\psi(x)=|x|^{-q}$ at the origin. Our goal is to show
that $v_{\lambda_0}(x)\equiv 0$ in $\Sigma_{\lambda_0}$. Otherwise,
by the strong maximum principle in Lemma \ref{hop},  we have
$v_{\lambda_0}>0$ in $\Sigma_{\lambda_0}$. If this is the case, we
will show that the plane can continue to move  to the right a little
bit more, that is, there exists a $\epsilon_0$ such that, for all
$0<\epsilon<\epsilon_0$, we have
\begin{equation}
v_{\lambda_0+\epsilon}\geq 0, \quad \forall x\in
\Sigma_{\lambda_0+\epsilon}. \label{arg}
\end{equation}
It contradicts the definition of $\lambda_0$. Therefore, (\ref{cri})
must be true. Set
$$\bar
v_{\lambda}(x):=\frac{v_{\lambda}(x)}{\psi(x)}.$$
 Suppose that (\ref{arg}) does not hold, then
there exist a sequence of $\epsilon_i$ such that $\epsilon_i\to 0$
and a sequence of $\{x^i\}$,  where $\{x^i\}$ is the minimum point
such that $$\bar
v_{\lambda}(x)=\liminf_{\Sigma_{\lambda_0+\epsilon_i}}v_{\lambda}(x).$$
 We claim that there exists a $\bar R$ such that
$|x^i|<\bar R$ for all $i$. For a clear presentation, this claim is
verified in Lemma \ref{lemq} below. By the boundedness of $\{x^i\}$,
there exists a subsequence of $\{x^i\}$ converging to some point
$x^0\in \Sigma_{\lambda_0}$. Since
$$\bar v_{\lambda_0}(x^0)=\lim_{i\to\infty}\bar v_{\lambda_0+\epsilon_i}(x_i)\leq 0$$
and $\bar v_{\lambda_0}(x)>0 $ for $x\in \Sigma_{\lambda_0}$, we
obtain that $x^0\in T_{\lambda_0}$ and $\bar v_{\lambda_0}(x^0)=0$.
By the regularity of fully nonlinear equations in Lemma \ref{regu}
and the fact that $\psi(x)\in C^2(\Sigma_{\lambda_0})$, we know that
at least $\bar v_{\lambda}(x)\in C^1(\Sigma_{\lambda_0}).$
Consequently,
$$ \nabla \bar v_{\lambda_0}(x^0)=\lim_{i\to\infty}\nabla \bar
v_{\lambda_0+\epsilon_i}(x^i)=0.$$ It follows that
\begin{equation}
\nabla v_{\lambda_0}(x^0)=\nabla \bar
v_{\lambda_0}(x^0)\psi(x^0)+\bar v_{\lambda_0}(x^0)\nabla
\psi(x^0)=0 \label{gra}.
\end{equation}
Since $v_{\lambda_0}(x^0)=0$ and $v_{\lambda_0}(x)>0$ for $x\in
\Sigma_{\lambda_0}$, thanks to  the Hopf lemma (i.e. Lemma
\ref{hop}) , we readily get that
$$\frac{\partial v_{\lambda_0}}{\partial n}(x^0)<0,$$
where $n$ is the outward normal at $T_{\lambda_0}$. It is a
contradiction to (\ref{gra}). In the end, we conclude that
$u_{\lambda_0}(x)\equiv u(x)$, i.e. (\ref{cri}) holds.

\end{proof}
The following lemma verifies the claim in the proof of Theorem
\ref{th1}.
\begin{lemma}
There exists a $\bar R$ (independent of $\lambda$) such that
$|x_0|<\bar R$, where $x_0$ is the point where $\bar v_{\lambda}(x)$
achieves the minimum and $\bar v_{\lambda}(x_0)<0$. \label{lemq}
\end{lemma}
\begin{proof}
If $|x_0|$ is sufficiently large, by the decay rate of $u$,
\begin{equation}
c(x_0)=p\psi_{\lambda}^{p-1}(x_0)<
C|x_0|^{-2}=-\frac{\mathcal{M}^{+}_{\Lambda_1, \Lambda_2}(D^2
\psi)(x_0)}{\psi(x_0)},
\label{une}
\end{equation}
where $C=-q(\Lambda_2(q+1)-\Lambda_1(n-1))>0$ and
$\psi(x)=|x|^{-q}$. It follows from the argument of Theorem
\ref{max} in the case of $\gamma=0$ that
$$
\mathcal{M}_{\Lambda_1, \Lambda_2}^-(D^2\bar v_{\lambda})- \bar b
(x)|D\bar v_{\lambda}|+\bar c(x)\bar v_{\lambda} \leq 0 \quad
\mbox{in} \ \Sigma_{\lambda}.
$$
Here $$\bar b(x)=\frac{2\sqrt{n}\Lambda_2 |D\psi|}{\psi}$$ and
$$\bar c(x)=c(x)+\frac{\mathcal{M}_{\Lambda_1, \Lambda_2}^+(D^2
\psi)}{\psi}.$$ From (\ref{une}), we see that there exists a
neighborhood $\Omega'$ of $x_0$ such that $\bar c(x)<0$ in $\Omega'
$. The strong maximum principle in Lemma \ref{str} further implies
that
\begin{equation}\bar v_{\lambda}(x)\equiv \bar v_{\lambda}(x_0)<0
\quad \mbox{for} \ |x|>|x_0|.
\label{any}
\end{equation} On the other
hand,
$$\bar
v_{\lambda}(x)=[o(|x_\lambda|^{-\frac{2}{p-1}})-o(|x|^{-\frac{2}{p-1}})|x|^q]\to
0$$ as $|x|\to \infty$,  which contradicts (\ref{any}). Hence the
lemma is completed.
\end{proof}

\begin{proof}[Proof of Corollary \ref{coro}:]
Adopting the same notations in the proof of Theorem \ref{th1}, for
the general function $f(u)$, we have
\begin{equation}
\mathcal{M}^{-}_{\Lambda_1, \Lambda_2}(D^2
v_{\lambda})+c_{\lambda}(x)v_{\lambda}(x)\leq 0, \label{sub2}
\end{equation}
where
\begin{equation}
 c_\lambda(x)= \left\{ \begin{array}{lll}
\frac{f(u_\lambda(x))-f(u(x))}{u_\lambda(x)-u(x)},  \quad \
&\mbox{if} \
u_\lambda(x)\not= u(x), \\
\\
0, \quad &\mbox{otherwise}.
\end{array}
\right.
\end{equation}
As argued in Theorem \ref{th1}, we should verify that
\begin{equation}
c_{\lambda}(x)\leq \frac{C}{|x|^2} \label{www}
\end{equation} and
\begin{equation}
\liminf_{|x|\to\infty}v_{\lambda}(x)|x|^q\geq 0. \label{lim2}
\end{equation}
We only need to focus on the points $\tilde{x}$ where
$u_{\lambda}(\tilde{x})< u (\tilde{x})$ for (\ref{www}). From the
assumption (\ref{assf}), if $\tilde{x}$ is large enough,
\begin{equation}
c_{\lambda}(\tilde{x}) \leq
c(|u_\lambda|+|u|)^{\alpha}(\tilde{x})=O(|\tilde{x}|^{(2-n^{\ast})\alpha})\leq
\frac{C}{|\tilde{x}|^2}
\end{equation}
for $\alpha>\frac{2}{n^{\ast}-2}.$ Recall that
$n^\ast=\frac{\Lambda_1}{\Lambda_2}(n-1)+1$. Since $u(x)$ is
positive, then
$$v_{\lambda}(x)|x|^q>-u(x)|x|^q.$$
If $u(x)=O(|x|^{2-n^\ast})$, then
$$\liminf_{|x|\to\infty}v_{\lambda}(x)|x|^q\geq \liminf_{|x|\to
\infty}-u(x)|x|^q=0$$ for $0<q< n^\ast-2$. Hence (\ref{www}) and
(\ref{lim2}) are satisfied. The rest of proof follows from the same
argument in Theorem \ref{th1}.
\end{proof}

\section{Symmetry of Viscosity Solutions in a Punctured Ball}
In this section, we consider the radial symmetry of viscosity
solutions in a punctured ball. Due to the singularity of the point,
the corresponding maximum principle shall be established. Instead of
only considering sufficiently small domains, our result is valid under
the appropriate  upper bound of $c(x)$. The result also holds  if $c(x)$ is bounded
and the domain is appropriately small. Thanks to Lemma \ref{diff}, we
only consider the following equation.
\begin{equation}
\mathcal{M}_{\Lambda_1, \Lambda_2}^-(D^2 u)-\gamma|Du|+c(x)u\leq 0
\quad \ \mbox{in} \ \Omega\backslash \{0\}. \label{sin}
\end{equation}

\begin{lemma}
Let  $\Omega$ be a connected and bounded domain in $\mathbb R^n$ and $u$
be the viscosity solution of (\ref{sin}). Assume that $c(x)\in
L^{\infty}(\Omega\backslash\{0\})$, and
\begin{equation}
\left \{ \begin{array}{lll} c(x)\leq
\frac{q(\Lambda_1(n-1)-\Lambda_2(q+1))}{|x|^2}-\frac{\gamma
q}{|x|} \ \mbox{with}\ 0<q<n^\ast-2 \quad &\mbox{if} \ n^{\ast}>2, \\
\vspace{2mm}
 or \\
 \vspace{2mm}
c(x)\leq {\Lambda_2}/4 (-\ln|x|)^{-2}|x|^{-2}-\gamma/2
(-\ln|x|)^{-1}|x|^{-1} \ \mbox{with}\ |x|\leq 1 \ \mbox{in} \ \Omega
\quad &\mbox{if} \ n^{\ast}=2.
\end{array}
\right. \label{ccc}
\end{equation}
Moreover, $u$ is bounded from below and $u\geq 0$ on $\partial\Omega$.
Then $u\geq 0$ in $\Omega\backslash \{0\}.$ \label{sma}
\end{lemma}

\begin{proof}
Our proof is based on the idea in Theorem \ref{max}. Recall again
that $n^{\ast}=\frac{\Lambda_1}{\Lambda_2}(n-1)+1$. If $n^\ast>2$,
let $\psi(x)=|x|^{-q}$. If $n^\ast=2$, we select
$\psi(x)=(-\ln|x|)^a$, where $0<a<1$.  Set $$\bar
u(x):=\frac{u(x)}{\psi(x)}.$$ Since $u$ is bounded from  below in
$\Omega\backslash \{0\}$ and $\psi(x)\to \infty $ as $|x|\to 0$,
then
$$\liminf_{|x|\to 0} \bar u(x)\geq 0.$$
It is easy to know that  $\bar u(x)\geq 0$ on $\partial \Omega$.
Suppose $u(x)<0$ somewhere in $\Omega\backslash \{0\}$, then $\bar
u(x)<0$ somewhere in $\Omega\backslash \{0\}$. Hence
$\inf_{\Omega\backslash \{0\}}\bar u(x)$ is achieved at some point
$x_0\in \Omega\backslash \{0\}$. Therefore, we can find a
neighborhood $\Omega'$ of $x_0$ such that $\bar u(x)<0$ and $\bar
u(x)\not \equiv \bar u(x_0)$ in $\Omega'$. Otherwise, $\bar
u(x)\equiv \bar u(x_0)$ in $\Omega'$, which is obviously impossible.
Recall in Theorem \ref{max} that,
\begin{equation}
\mathcal{M}_{\Lambda_1, \Lambda_2}^-(D^2\bar u)- \bar b (x)|D\bar
u|+\bar c(x)\bar u\leq 0 \quad \mbox{in} \ \Omega',
\end{equation}
where
$$\bar
c(x)=c(x)+\frac{\mathcal{M}_{\Lambda_1, \Lambda_2}^+(D^2
\psi)+\gamma|D\psi|}{\psi}.$$

In order to apply the strong maximum principle, we need $\bar c(x)\leq
0$, i.e.
\begin{equation} c(x)\leq -\frac{\mathcal{M}_{\Lambda_1,
\Lambda_2}^+(D^2 \psi)+\gamma|D\psi|}{\psi}. \label{atl}
\end{equation} If
$n^{\ast}>2$, then $\psi(x)=|x|^{-q}$,
 $$
\frac{\mathcal{M}_{\Lambda_1, \Lambda_2}^+(D^2
\psi)+\gamma|D\psi|}{\psi}=\frac{q(\Lambda_2(q+1)-\Lambda_1(n-1))}{|x|^2}+\frac{\gamma
q}{|x|}.$$ Let
$$c(x)\leq \frac{q(\Lambda_1(n-1)-\Lambda_2(q+1))}{|x|^2}-\frac{\gamma
q}{|x|}.
$$
Then (\ref{atl}) is satisfied.

 If $n^{\ast}=2$, then $\psi(x)=(-\ln|x|)^a$,
$$\frac{\mathcal{M}_{\Lambda_1, \Lambda_2}^+(D^2
\psi)+\gamma|D\psi|}{\psi}=\Lambda_2(a-1)a(-\ln|x|)^{-2}|x|^{-2}+\gamma
a(-\ln|x|)^{-1}|x|^{-1}.
$$
Hence  we may assume that
\begin{equation}
c(x)\leq {\Lambda_2}/4 (-\ln|x|)^{-2}|x|^{-2}-\gamma/2
(-\ln|x|)^{-1}|x|^{-1}, \label{ass}
\end{equation}
which implies that (\ref{atl}) holds for $a=1/2$.  If $c(x)$ is in
the above range, by the strong maximum principle in Lemma \ref{str},
we readily deduce that $u(x) \equiv u(x_0)$ in $\Omega'$. We then
arrive at a contradiction. The proof of the lemma follows.
\end{proof}

\begin{remark}
The assumption (\ref{ccc}) is clearly satisfied when $|c(x)|$ is
bounded and $\Omega$ is sufficiently small.
\end{remark}

With the above maximum principle in hand, we are able to prove the
radial symmetry of viscosity solutions. We adapt the argument of
\cite{CLN1} in our setting. Let the domain $O$ be bounded and convex
in direction of $x_1$, symmetric with respect to the hyperplane
$\{x_1=0\}$. We prove the radial symmetry and monotonicity
properties in $O$. Theorem \ref{th2} is an immediate consequence of
Theorem \ref{sss} below. Let us first introduce several notations.
Set
$$\Sigma_{\lambda}:=\{x=(x_1,\cdots, x_n)\in O|
x_1<\lambda\}$$ and $T_\lambda=\{x\in O| x_1=\lambda \}$. Define
$x^{\lambda}$ be the reflection of $x$ with respect to $T_\lambda$.
Let $$u_\lambda(x)=u(x^{\lambda})$$ and
$$v_\lambda(x)=u_\lambda(x)-u(x).$$

\begin{theorem}
Let $u\in C(\bar O \backslash \{0\})$ be a positive viscosity
solution of
\begin{equation}
F(D u, D^2 u)+f(u)=0 \quad  \quad \mbox{in} \ \bar O \backslash
\{0\}.\label{sph}
\end{equation}
Then $u$ is symmetric in $x_1$, that is, $u(x_1,x_2,\cdots,
x_n)=u(-x_1,x_2,\cdots, x_n)$ for all $x\in \ O \backslash \{0\}$.
In addition, $u$ is strictly increasing in $x_1<0$. \label{sss}
\end{theorem}

\begin{proof}[Proof of Theorem \ref{sss}] Without loss of
generality, we may assume that $\inf_O x_1=-1$. We carry out the
moving plane method in two steps.

\emph{Step 1}: We show that the plane can move, i.e. there exists
$-1<\lambda_0<-\frac{1}{2}$ such that
$$v_{\lambda}\geq 0 \ \ \mbox{in} \ \Sigma_{\lambda} $$
for $-1<\lambda <\lambda_0$. By (F2), $u_\lambda$ satisfies the same
equation as $u$ does.  Thanks to Lemma \ref{diff}, we know that
$v_{\lambda}$ satisfies
\begin{equation}
\mathcal{M}_{\Lambda_1, \Lambda_2}^-(D^2
v_{\lambda})-\gamma|Dv_{\lambda}|+c_\lambda(x)v_{\lambda}\leq 0
\quad \ \mbox{in} \ O \backslash \{0\}, \nonumber
\end{equation}
where
\begin{equation}
 c_\lambda(x)= \left\{ \begin{array}{lll}
\frac{f(u_\lambda(x))-f(u(x))}{u_\lambda(x)-u(x)},  \quad \
&\mbox{if} \
u_\lambda(x)\not= u(x), \\
\\
0, \quad &\mbox{otherwise}.
\end{array}
\right.
\end{equation}
 Since $f$ is locally Lipschitz in $(0,  \infty)$, then
$|c_\lambda(x)|<C$ in $O$ for some $C>0$. It is clear that
$v_{\lambda}(x)\geq 0$ in $\partial \Sigma_{\lambda}$. If $\lambda$
is sufficient close to $-1$, then $\Sigma_{\lambda}$ is small
enough. By the maximum principle for small domains in Lemma
\ref{min}, we readily deduce that $v_{\lambda}\geq 0$ in
$\Sigma_{\lambda}$. Step 1 is then completed.

 Define
$$\lambda_0=\sup\{ \lambda|-1<\mu<0, \ v_{\mu}\geq 0 \ \mbox{in} \
\Sigma_{\mu} \backslash \{0^\mu\} \ \mbox{for} \
\mu\leq\lambda<0\}.$$

\emph{Step 2}: We are going to show that $\lambda_0=0$. If it is
true, we move the plane from the position where $\sup_O x_1=1$ to
the left. By the symmetry of $O$, the plane will reach $\lambda_0=0$
again. Hence the symmetry of viscosity solutions is obtained. We
divide the proof into three cases and  show that the following cases
are impossible to occur.

\emph{Case 1}: $-1<\lambda_0<-\frac{1}{2}$.

If this is the case, we are going to show that the plane can still
be moved a little bit more to the right.  By the strong maximum principle
in Lemma \ref{hop}, we have $v_{\lambda_0}(x)>0$ in
$\Sigma_{\lambda_0}$. Set $\lambda=\lambda_0+\epsilon$ for
 sufficiently small $\epsilon$. Let $K$ be a compact subset in
$\Sigma_{\lambda_0}$ such that $|\Sigma_{\lambda_0}\backslash
K|<{\delta/2}$. Recall that $\delta$ is the measure of $O$ for which
the maximum principle for small domains in Lemma \ref{min} holds. By
the continuity of $v_{\lambda}$, there exists some $r >0$ such that
$v_{\lambda}>r$ in $K$. In the remaining $\Sigma_{\lambda}\backslash
K$, we can check that $v_\lambda$ satisfies
\begin{equation}
\left \{\begin{array}{lll}
 \mathcal{M}_{\Lambda_1, \Lambda_2}^-(D^2
v_{\lambda})-\gamma|Dv_{\lambda}|+c_\lambda(x)v_{\lambda}\leq 0
\quad \ &\mbox{in} \ \Sigma_{\lambda}\backslash K , \nonumber \\
v_{\lambda}\geq 0 \quad \ &\mbox{on} \
\partial(\Sigma_{\lambda}\backslash K).
\end{array}
\right.
\end{equation}
By the maximum principle for small domains again, $v_{\lambda}\geq 0$
in $\Sigma_{\lambda}\backslash K$ by selecting sufficiently small
$\eps$. Together with the fact that $v_{\lambda}\geq r$ in $K$, we
infer that $v_{\lambda}\geq 0$ in $\Sigma_{\lambda}$. It contradicts
the definition of $\lambda_0$.

\emph{Case 2}: $\lambda_0=-\frac{1}{2}$.

We also argue that the plane can be moved further, which indicates
that $\lambda_0=-\frac{1}{2}$ is impossible. Since $O$ is symmetric
with respect to the hyperplane $x_1=0$, then $0^{-1/2}=(-1,
0,\cdots, 0).$ We select a compact set $K$ in $\Sigma_{-1/2}$ such
that $|\Sigma_{\lambda_0}\backslash K|<\delta/2$. By the positivity
and continuity of  $v_{-1/2}$,  there exists some $r>0$ such that
$v_{-1/2}>r$ in $K$.  Without loss of generality, we may assume that
$dist (K, \Sigma_{-1/2})\geq r'$ for some $r'>0$. We consider  a
small ball $\mathbb B_{{r'}/2}(e)$ centered at $e=(-1, 0, \cdots,
0)$ with radius ${r'}/2$. From the positivity of $v_{-1/2}$ again,
we have, making $r$ smaller if necessary,
$$v_{-1/2}>r/2  \ \mbox{in} \
\partial\mathbb B_{{r'}/2}(e)\cap \bar O.$$ Let
$\lambda=-1/2+\epsilon$ for small $\epsilon>0$. By the continuity of
$v_\lambda$, we get
$$v_\lambda>r/4 \ \mbox{in} \ (\partial \mathbb
B_{{r'}/2}(e)\cap \bar O)\cup K.$$ For such small $\epsilon$,
$0^{-1/2+\epsilon}$ lies in $\mathbb B_{{r'}/2}(e)\cap \bar O$. We
also know that $v_{\lambda}\geq 0$ on $\mathbb B_{{r'}/2}(e)\cap
\partial O$. Therefore,  $$v_{\lambda}\geq 0  \ \mbox{in}  \ \partial(\mathbb B_{{r'}/2}(e)\cap
 O).$$
Choosing $r'$ so small that Lemma \ref{sma} is valid, then
$$v_{\lambda}\geq 0 \ \mbox{in} \ \mathbb B_{{r'}/2}(e)\cap O.$$ We
consider the remaining set $\Sigma_{\lambda}\backslash (K\cup \mathbb
B_{{r'}/2}(e))$. we can verify that $v_{\lambda}$ satisfies
\begin{equation}
\left \{\begin{array}{lll}
 \mathcal{M}_{\Lambda_1, \Lambda_2}^-(D^2
v_{\lambda})-\gamma|Dv_{\lambda}|+c_\lambda(x)v_{\lambda}\leq 0
\quad \ &\mbox{in} \ \Sigma_{\lambda}\backslash (K\cup \mathbb B_{{r'}/2}(e)) , \nonumber \\
v_{\lambda}\geq 0 \quad \ &\mbox{on} \
\partial(\Sigma_{\lambda}\backslash (K\cup \mathbb B_{{r'}/2}(e))).
\end{array}
\right.
\end{equation}
Therefore, for sufficiently small $\epsilon$, the maximum principle of
small domains implies that $v_{\lambda}\geq 0$ in
$\Sigma_{\lambda}\backslash (K\cup \mathbb B_{{r'}/2})$. In
conclusion, $v_{\lambda}\geq 0$ for $\lambda=-1/2+\epsilon$. We
arrive at a contradiction.

\emph{Case 3}: $-1/2<\lambda_0<0$.

We show that this critical position is also impossible. For the
singular point $0^{\lambda_0}$, we choose a ball  $\mathbb
B_{{r'}/2}(0^{\lambda_0})$ centered at $0^{\lambda_0} $ with radius
${r'}/2$. Let $\lambda=\lambda_0+\epsilon$. For $\epsilon>0$ small
enough, $0^{\lambda}$ still lies in  $\mathbb
B_{{r'}/2}(0^{\lambda_0})$. By the continuity and positivity of
$v_{\lambda_0}$, there exists some $r> 0$ such that $v_{\lambda}\geq
r$ on $\partial \mathbb B_{{r'}/2}(0^{\lambda_0})$. Applying Lemma
\ref{sma} for small value of  $r'/2$, we infer that $v_{\lambda}\geq
0$ in $\mathbb B_{{r'}/2}(0^{\lambda_0})$. Similar argument as Case
1 and Case 2 could show that $v_{\lambda}\geq 0$ in
$\Sigma_{\lambda}\backslash\{0\}$ for $\lambda=\lambda_0+\epsilon.$

\end{proof}

\section{The Radial Symmetry for Viscosity Solutions of Fully Nonlinear Parabolic Equations}

We consider the  radial symmetry of fully nonlinear parabolic
equation in this section. we first show that the difference of
supersolution and subsolution of the parabolic equation satisfies an
inequality involving Pucci extremal operator, which enables us to
compare the value of $u$ at $x$ and its value at the reflection of
$x$. The following lemma  is non trivial since $u$  is not of class
$C^2$. The proof of  the lemma below is inspired by the work in
\cite{DS} and \cite{MQ}.

\begin{lemma}
Let $u_1, u_2$ be  a continuous subsolution and supersolution
respectively in $\mathbb R^n\times (0, \ T]$ of
\begin{equation}
\partial_t u- F(D u, D^2 u)-f(u)=0.
\label{par}
\end{equation}
 Then
$\tilde{w}=u_2-u_1$ is a viscosity supersolution of
\begin{equation}
-\partial_t \tilde{w}+\mathcal {M}_{\Lambda_1, \Lambda_2}^-(D^2
\tilde{w})-\gamma |\nabla \tilde{w}|+ c(x, t)\tilde{w}\leq 0,
\label{vis}
\end{equation}
where \begin{equation}
 c(x, t)= \left\{ \begin{array}{lll}
\frac{f(u_1(x, t))-f(u_2(x, t))}{u_1(x, t)-u_2(x, t)},  \quad \
&\mbox{if} \
u_1(x, t)\not= u_2(x, t), \\
\\
0, \quad &\mbox{otherwise}.
\end{array}
\right. \label{god}
\end{equation}
\label{ppp}
\end{lemma}

\begin{proof}
We consider $w=u_1-u_2=\tilde{w}$, then apply the property of
$\mathcal {M}_{\Lambda_1, \Lambda_2}^-(D^2 w)=-\mathcal
{M}_{\Lambda_1, \Lambda_2}^+(D^2 \tilde{w})$ to verify (\ref{vis}).
Let $\varphi\in C^2$ be a test function such that $w-\varphi$ has a
local maximum at $(\tilde{x}, \tilde{t})$. Then there exists $r>0$
such that, for all $(x, t)\in \overline{\mathbb
B}_r(\tilde{x})\times (\tilde{t}-r,\, \tilde{t}]\subset \mathbb
R^n\times (0, \, T]$, $(w-\varphi)(x,t)<(w-\varphi)(\tilde{x},
\tilde{t})$. Define
$$
\Phi_{\epsilon}(x, y, t)=u_1(x, t)-u_2(y, t)-\varphi(x,
t)-\frac{|x-y|^2}{\eps^2}.$$ Let $(x_\eps, y_\eps, t_\eps)$ be the
maximum point of $\Phi_{\epsilon}(x, y, t)$ in $\overline{\mathbb
B}_r(\tilde{x})\times \overline{\mathbb
B}_r(\tilde{x})\times(\tilde{t}-r,\, \tilde{t}]$. Standard argument
shows that
\begin{equation}
\left \{ \begin{array}{ll}
 (i): (x_\eps, y_\eps)\to (\tilde{x}, \tilde{x}),
\\
\\
(ii): \frac{|x_\eps-y_\eps|^2}{\eps^2}\to 0
\end{array}
\right.
\end{equation}
 as $\eps\to 0$. Let
$\theta=\overline{\mathbb B}_r(\tilde{x})$ and $$\psi_{\eps}(x, y,
t)=\varphi(x, t)+\frac{|x_\eps-y_\eps|^2}{\eps^2}.$$ The argument of
Theorem 8.3 in \cite{CIL} indicates that, for all $\alpha>0$, there
exist $X, Y\in S_n(\mathbb R^n)$ such that
\begin{equation}
\left \{ \begin{array}{lll}
 (i): (a_\eps, D_x
\psi_{\eps}( x_\eps, y_\eps, t_\eps), X)\in \overline {\mathcal
{P}_\theta}^{2, +} u_1(x_\eps, t_\eps),  \\
\\
\indent  \ \ (b_\eps, D_y \psi_{\eps}( x_\eps, y_\eps, t_\eps),
Y)\in
\overline {\mathcal {P}_\theta}^{2, +} (-u_2)(y_\eps, t_\eps),  \\
\\
(ii): -(1/\alpha+\|A\|)Id\leq \left( \begin{array}{ccc} X & 0 \\
0 & Y
\end{array}
\right) \leq A+\alpha A^2,  \\
\\
(iii): a_\eps+b_\eps=\partial_t\psi_\eps( x_\eps, y_\eps,
t_\eps)=\partial_t \varphi(x_\eps, t_\eps),
\end{array}
\right. \label{last}
\end{equation}
where $A=D^2 \psi_\eps( x_\eps, y_\eps, t_\eps)= \left(
\begin{array}{ccc}
D^2_x\varphi(x_\eps, t_\eps)+\frac{2}{\eps^2}Id &
-\frac{2}{\eps^2}Id
\\
\\
-\frac{2}{\eps^2}Id & \frac{2}{\eps^2}Id \end{array} \right). $ \\
Furthermore, by the definition of ${\mathcal {P}_\theta}^{2, +},
{\mathcal {P}_\theta}^{2, -}$, we have
\begin{equation}
a_\eps-F(D_x\psi_\eps(x_\eps, y_\eps, t_\eps), X)- f(u_1(x_\eps,
t_\eps))\leq 0, \label{ar1}
\end{equation}
\begin{equation}
-b_\eps-F(-D_y\psi_\eps(x_\eps, y_\eps, t_\eps), -Y)- f(u_2(y_\eps,
t_\eps))\geq 0. \label{ar2}
\end{equation}
Combining $(iii)$ in  (\ref{last}), (\ref{ar1}) and (\ref{ar2}), we
obtain
\begin{equation}
\partial_t \varphi (x_\eps, t_\eps)-F(D_x\psi_\eps(x_\eps, y_\eps, t_\eps),
X)+F(-D_y\psi_\eps(x_\eps, y_\eps, t_\eps), -Y)-f(u_1(x_\eps,
t_\eps))+f(u_2(y_\eps, t_\eps))\leq 0.\label{cop}
\end{equation}
Let $\alpha=\eps^2$. A similar argument to that  in \cite{DS} leads to
\begin{equation}
X-D^2\varphi_\eps(x_\eps, t_\eps)+Y\leq -C\eps^2 Y^2+O(\eps)
\end{equation}
for some $C>0$. Then
\begin{equation}
\begin{array}{lll}
(\partial_t \varphi-\mathcal{M}^+_{\Lambda_1,
\Lambda_2}(D^2\varphi)-f(u_1))(x_\eps, t_\eps)+f(u_2)(y_\eps,
t_\eps)&-&\gamma |D_x\psi+D_y\psi|(x_\eps, y_\eps, t_\eps) \nonumber
\\
&+&C\eps^2\mathcal{M}^-_{\Lambda_1, \Lambda_2}(Y^2)+O(\eps)\leq 0.
\end{array}
\end{equation}

Since $\mathcal{M}^-_{\Lambda_1, \Lambda_2}(Y^2)\geq 0$,  letting
$\eps\to 0$, then
$$ (\partial_t \varphi-\mathcal{M}^+_{\Lambda_1,
\Lambda_2}(D^2\varphi)-\gamma |D_x\varphi|-f(u_1)+f(u_2))(\tilde{x},
\tilde{t})\leq 0.
$$
By the mean value theorem,
$$ (\partial_t \varphi-\mathcal{M}^+_{\Lambda_1,
\Lambda_2}(D^2\varphi)-\gamma |D_x\varphi|-c(\tilde{x},
\tilde{t})(u_1-u_2))(\tilde{x}, \tilde{t})\leq 0,
$$ where $c(x,t)$ is in (\ref{god}).
Hence
$$ \partial_t w-\mathcal{M}^+_{\Lambda_1,
\Lambda_2}(D^2 w )-\gamma |D_x w|-c(x, t)w\leq 0$$ for $(x, t)\in
\mathbb R^n\times (0, T]$. Since $\tilde{w}=-w$,
$$ -\partial_t \tilde{w}+\mathcal{M}^-_{\Lambda_1,
\Lambda_2}(D^2 \tilde{w} )-\gamma |D_x \tilde{w}|+c(x,
t)\tilde{w}\leq 0.$$ The proof of the lemma is then fulfilled.
\end{proof}

We are ready to give the proof of  Theorem \ref{th3}.

\begin{proof}[Proof of Theorem \ref{th3}]
We adopt the moving plane method to prove the theorem. Define
$$ \Sigma_{\lambda}=\{(x_1, x', t)\in \mathbb R^{n+1}|x_1<\lambda,
0<t \leq T\},$$ where $x'= \{x_2,\cdots, x_n\}$.  Set
$$ u_\lambda(x_1, x', t)=u(2\lambda-x_1, x', t) \ \mbox{and} \ v_{\lambda}(x, t)=u_\lambda(x, t)-u(x, t). $$

\emph{Step 1:} We start the plane from negative infinity. Since
$u_\lambda$ satisfies the same equation as $u$ does by \emph{(F2)}.
Thanks to Lemma \ref{ppp}, we have
$$
-\partial_t v_{\lambda}+\mathcal {M}_{\Lambda_1, \Lambda_2}^-(D^2
v_\lambda)-\gamma |\nabla v_\lambda|+ c(x,t) v_\lambda\leq 0.
$$
We may assume that $|c(x, t)|\leq c_0$ for some $c_0>0$, since
$f(u)$ is locally Lipschitz. Let
$$\bar v_{\lambda}=\frac{v_\lambda}{e^{-(c_0+1)t}},$$
then $\bar v_{\lambda}$ satisfies
\begin{equation}
-\partial_t \bar v_{\lambda}+\mathcal {M}_{\Lambda_1,
\Lambda_2}^-(D^2 \bar v_\lambda)-\gamma |\nabla \bar v_\lambda|+
\tilde{c}(x,t)\bar v_\lambda\leq 0,
\end{equation}
where $\tilde{c}(x,t)=c(x,t)-c_0-1$. Note that $\tilde{c}(x,t)<0$.
In order to prove that $ v_{\lambda}\geq 0$ in $\Sigma_{\lambda},$
it is sufficient to show that $\bar v_\lambda\geq 0$ in
$\Sigma_{\lambda}$. Suppose the contrary, that $\bar v_\lambda <0 $
somewhere in $\Sigma_{\lambda}$. Since
$$ |u(x,t)|\to 0 \  \mbox{uniformly as } |x|\to \infty,$$
then
$$v_\lambda(x, t)\geq -u(x, t) e^{(c_0+1)t} \to 0$$
as $|x|\to \infty$. Due to  the fact that $\bar v_\lambda=0$ on
$\partial \Sigma_\lambda:= \{(x_1, x', t)| x_1=\lambda, 0<t\leq T\}$
and the assumption of initial boundary condition $u_0(x)$, there
exists some point $z^0\in \Sigma_\lambda$ such that
$$\bar v_{\lambda}(z^0)=\min_{z\in \Sigma_\lambda} \bar
v_{\lambda}(x,t)<0.$$ By the strong maximum principle for fully nonlinear
parabolic equations, we know it is a contradiction. Step 1 is then completed.

\emph{Step 2:} Set
$$\lambda_0:=\sup\{\lambda<0 | v_{\mu}\geq 0 \  \mbox{in} \ \Sigma_{\mu} \ \mbox{for} -\infty< \mu<\lambda\}. $$
Our goal is to show that $\lambda_0=0$. Suppose that $\lambda_0<0$,
then there exists sufficiently small $\eps>0$ such that
$\lambda_0+\eps<0$. We are going to prove that $v_{\lambda}\geq 0$
in $\Sigma_{\lambda}$ for $\lambda=\lambda_0+\eps$, which
contradicts the definition of $\lambda_0$.  If
$v_{\lambda_0+\eps}<0$ somewhere in $\Sigma_{\lambda_0+\eps}$, by
the asymptotic behavior of $u$ and the initial boundary condition,
we know that the minimum point is achieved in the interior of
$\Sigma_{\lambda_0+\eps}$. By the same argument as that in Step 1,
we see it is impossible. Therefore, we confirm that  $\lambda_0=0$,
that is, $u$ is nondecreasing in $x_1$ and $u(x_1, x', t)\leq
u(-x_1, x', t)$ for $x_1\leq 0$.

If the initial value $u_0$ is radial symmetry and nonincreasing in
$|x|$. We move the plane from positive infinity to the left. By the
same argument as above, we will reach at $\lambda_0=0$ again, which
leads to the symmetry of the solution at $x_1=0$. By the rotation
invariance of the equation, we obtain that $u$ is radially symmetric
with respect to $(0, t)$ for any fixed $t\in (0, T]$ and
nonincreasing in $|x|$.

\end{proof}

\end{document}